\documentclass[11pt]{amsart}
\usepackage[T1]{fontenc}
\usepackage{amsmath,amssymb,amsthm,mathtools,booktabs}
\usepackage[expansion=false]{microtype}
\usepackage[hidelinks]{hyperref}
\hypersetup{pdftitle={A Fricke transformation for cubic-residue eta products of level 13},
 pdfauthor={Redacted}}
\newtheorem{theorem}{Theorem}[section]
\newtheorem{proposition}[theorem]{Proposition}
\newtheorem{lemma}[theorem]{Lemma}
\newtheorem{corollary}[theorem]{Corollary}
\theoremstyle{remark}
\newtheorem{remark}[theorem]{Remark}
\DeclareMathOperator{\ord}{ord}
\DeclareMathOperator{\divisor}{div}
\newcommand{\HH}{\mathbb H}
\newcommand{\QQ}{\mathbb Q}
\newcommand{\ZZ}{\mathbb Z}

\newcommand{\ttil}{\widetilde t}
\title[A Fricke transformation at level 13]{A Fricke transformation for cubic-residue eta products of level 13}

\author{Cetin Hakimoglu-Brown}
\subjclass[2020]{Primary 11F20; Secondary 11F27, 11P84, 11R16}
\keywords{Generalized Dedekind eta function, Fricke involution, Gaussian period, modular unit}
\date{}
\begin{document}
\begin{abstract}
We prove an explicit Fricke transformation for the two-dimensional space generated by normalized reciprocal products on the cubic-residue cosets modulo 13. The transformation matrix is a scalar multiple of a matrix of differences of cubic Gaussian periods. Its projective action is defined over the cyclic cubic field, whereas the normalized matrix is defined over the real cyclotomic field. The proof uses generalized eta functions, a modular unit of degree two on $X_1(13)$, and a five-coefficient identity. A separate divisor argument lifts the projective transformation to the asserted linear transformation. We also record the Fricke images at an arbitrary prime $p\equiv1\pmod3$ and prove that their normalized sine constants belong to the associated cyclic cubic subfield.
\end{abstract}
\maketitle

\section{Introduction}
Reciprocal products supported on residue classes are a familiar source of
$q$-series with modular transformation laws. Generalized eta functions and
Siegel functions provide a systematic way to transform their individual
factors~\cite{KL,Yang04}. A more specific question is whether a prescribed
small space of such products is preserved by a given transformation, and
whether its transformation matrix reflects the arithmetic of the residue
classes. We answer this question for the cubic-residue products modulo 13:
their natural normalizations span a two-dimensional space preserved by the
Fricke involution, with an explicit matrix expressed in cubic Gaussian
periods.

Put $q=e^{2\pi i\tau}$ for $\tau\in\HH$, and write
\[
 (a_1,\ldots,a_k;q)_\infty=\prod_{i=1}^k\prod_{n\ge0}(1-a_iq^n).
\]
All fractional powers mean $q^r=e^{2\pi ir\tau}$. The cubic residues modulo
13 and their two cosets are
\begin{equation}\label{eq:cosets}
 C_0=\{\pm1,\pm5\},\qquad C_1=\{\pm2,\pm3\},\qquad C_2=\{\pm4,\pm6\}.
\end{equation}
Define
\begin{equation}\label{eq:products}
 \begin{split}
 P_0(q)&=(q,q^5,q^8,q^{12};q^{13})_\infty^{-1},\\
 P_1(q)&=(q^2,q^3,q^{10},q^{11};q^{13})_\infty^{-1},\\
 P_2(q)&=(q^4,q^6,q^7,q^9;q^{13})_\infty^{-1}.
 \end{split}
\end{equation}
For $\zeta=e^{2\pi i/13}$, let
\[
 \eta_j=\sum_{a\in C_j}\zeta^a,\qquad
 \alpha=\eta_1-\eta_2,\quad \beta=\eta_0-\eta_2,\quad \gamma=\eta_1-\eta_0.
\]
The three normalized products satisfy
$q^{-1/6}P_0=q^{-1/6}P_1+q^{5/6}P_2$, by a specialization of the
Weierstrass theta identity (Lemma~\ref{lem:theta}). Thus two components
suffice for the following transformation law.

\begin{theorem}\label{thm:main}
For every $\tau\in\HH$, the vector
\[
 H(\tau)=\begin{pmatrix}q^{-1/6}P_0(q)\\q^{-1/6}P_1(q)\end{pmatrix}
\]
satisfies
\begin{equation}\label{eq:main}
 H\left(-\frac1{13\tau}\right)=S_{13}H(\tau),\qquad
 S_{13}=\frac1{\sqrt{13}}\begin{pmatrix}\gamma&\beta\\\alpha&-\gamma\end{pmatrix}.
\end{equation}
Moreover $S_{13}^2=I$ and $H(\tau+1)=e^{-\pi i/3}H(\tau)$.
\end{theorem}

The two scalar identities and their first matching coefficients are displayed
in Section~\ref{subsec:expansion}.

\subsection{Two model transformations}
The classical comparison is the normalized Rogers--Ramanujan pair
\[
 \mathbf R(\tau)=
 \begin{pmatrix}
 q^{-1/60}(q,q^4;q^5)_\infty^{-1}\\
 q^{11/60}(q^2,q^3;q^5)_\infty^{-1}
 \end{pmatrix}.
\]
Its inversion formula is
\begin{equation}\label{eq:RR-context}
 \mathbf R(-1/\tau)=\frac2{\sqrt5}
 \begin{pmatrix}
 \sin(2\pi/5)&\sin(\pi/5)\\
 \sin(\pi/5)&-\sin(2\pi/5)
 \end{pmatrix}\mathbf R(\tau);
\end{equation}
see, for example,~\cite[Theorem 1]{Ono}. This is a two-dimensional product
space with an explicit arithmetic inversion matrix.

A related example appears in Mizuno's study of Nahm sums for
symmetrizable matrices~\cite{Mizuno}. Put
\[
 B_a(q)=(q^a,q^3,q^6,q^{9-a};q^9)_\infty^{-1}\quad(a=1,2,4),
 \qquad
 \mathbf K_9(\tau)=
 \begin{pmatrix}q^{-1/18}B_1(q)\\q^{5/18}B_2(q)\\q^{11/18}B_4(q)\end{pmatrix}.
\]
The product-side form of~\cite[(9)]{Mizuno} is
\begin{equation}\label{eq:KR-context}
 \mathbf K_9(-1/\tau)=
 \begin{pmatrix}
 b_1&b_2&b_4\\b_2&-b_4&-b_1\\b_4&-b_1&b_2
 \end{pmatrix}\mathbf K_9(\tau/3),\qquad
 b_k=\frac1{2\sqrt3\sin(k\pi/9)}.
\end{equation}
Mizuno stated this formula without proof; for the corresponding
Kanade--Russell sums it would follow from sum--product identities that
remain open. It has since been proved on the product side: Wang and
Zhang~\cite[Theorem 1.1]{WZ} establish~\eqref{eq:KR-context} exactly as
displayed, and Mizuno has an independent proof via the $A_2$ Macdonald
identity~\cite[Remark 2]{WZ}. The coefficients are reciprocals of sines,
and the product modulus 9 is accompanied by the rescaling $\tau/3$.
Further explicit transformation formulas for generalized rank-two and
rank-three Nahm-sum vectors have been proved by Wang and
Wang~\cite{WW2,WW3}.

The proof in~\cite{WZ} is the closest precedent for the present paper. It
writes each component of $\mathbf K_9$ as the reciprocal of a product of
two generalized eta functions, applies Yang's transformation formula to
express the Fricke image as a $\zeta_9$-twisted product, and then
re-expands that product by Garvan's 3-dissection of the crank generating
function. The first two steps are the same as in Sections~2--3 below. The
third is where the methods diverge: at level 13 no comparable dissection
is available, and the re-expansion is instead obtained from the geometry
of $X_1(13)$, through a modular unit of degree two, a divisor bound, and
five explicit coefficients. Two further differences are structural. The
supports at modulus 9 overlap in $\{3,6\}$, whereas the three
sets~\eqref{eq:cosets} partition all nonzero residues modulo 13, which is
what makes the product relation of Lemma~\ref{lem:theta} and the scalar
lifting of Section~6 available; and the transformation matrix here is
expressed through cubic Gaussian periods, with the field-theoretic
consequence described next.

\subsection{The level-13 products and their arithmetic}
The products~\eqref{eq:products} are the product sides of three
conjectural index-$(1,13)$ Rogers--Ramanujan-type identities, all sharing
the quadratic form $2r^2+13rs+26s^2$, which is forced by a conductor-13
Rogers dilogarithm relation. One of them is
\begin{equation}\label{eq:motivation-conjecture}
 \sum_{r,s\ge0}
 \frac{q^{2r^2+13rs+26s^2}(1+q^{2r+13s+1})}
 {(q;q)_r(q^{13};q^{13})_s}
 \overset{?}{=}P_0(q),
\end{equation}
where $(a;q)_n=\prod_{j=0}^{n-1}(1-aq^j)$. Its two companions have the
same quadratic part and product sides $P_1$ and $P_2$. These conjectures
motivate the choice of products, but the present theorem concerns the
products alone and neither assumes nor proves the infinite sum--product
identities.

In contrast to the overlapping modulus-9 supports above, the three sets
$C_j$ partition all nonzero residues modulo 13. They index both the product
factors and the Gaussian periods in~\eqref{eq:main}. The matrix also has the
trigonometric expression
\begin{equation}\label{eq:sine-context}
 S_{13}=\frac14
 \begin{pmatrix}s_2^{-1}&s_1^{-1}\\s_0^{-1}&-s_2^{-1}\end{pmatrix},
 \qquad
 s_j=\prod_{\substack{a\in C_j\\1\le a\le6}}\sin\frac{\pi a}{13}.
\end{equation}
Each $s_j$ is a product of two sines. Quadratic Gauss-sum identities prove
that~\eqref{eq:main} and~\eqref{eq:sine-context} agree
(Lemma~\ref{lem:gauss}). In the period description, $\sqrt{13}S_{13}$ and the
induced projective transformation are defined over the cyclic cubic field
$K=\QQ(\eta_0)$, while $S_{13}$ itself is defined over
$K(\sqrt{13})=\QQ(\zeta)^+$. Thus the product supports and the projective
coefficients are controlled by the same cubic subfield, without requiring
the normalized matrix entries to belong to it.

\subsection{Theta identities, modular units, and the proof}
The level-13 theta literature includes Ramanujan's tredecic identities and
their proofs by O'Brien and Evans. Vasuki and Prabhu~\cite{VP} give new proofs
using the quintuple-product and Weierstrass identities, together with further
theta and partition applications. Our starting relation $P_0=P_1+qP_2$ is a
standard specialization of the latter identity. The additional assertion
here is that its two-dimensional normalized product space is preserved by
$W_{13}:\tau\mapsto-1/(13\tau)$, with the explicit matrix~\eqref{eq:main}.

There is also a geometric explanation for the projective part of the
answer. Modular-unit calculations on the genus-two curve $X_1(13)$ occur,
for example, in Brunault's work on Mahler measure~\cite{Brunault}. In the
present setting, the modular unit $t=qP_2/P_1$ has degree two on
$X_1(13)$, as the cusp calculation in Section~4 shows. It therefore gives a
hyperelliptic coordinate. Uniqueness
of the degree-two map up to a fractional linear change of coordinate implies
that the automorphism $W_{13}$ acts fractionally linearly on $t$. This
explains the form of the projective action; determining its coefficients in
the specified product normalization and lifting it to a linear action on
$H$ still require the calculations below.

The proof first computes the Fricke images by the standard transformation
of Siegel functions. A cusp-divisor bound then reduces the required ratio
identity to five coefficients, which are displayed explicitly. A separate
divisor calculation determines the common scalar and completes the linear
transformation. Finally, Section~7 records the Fricke images for every prime
$p\equiv1\pmod3$ and proves that the corresponding normalized sine constants
always lie in the cyclic cubic subfield. That general field statement is
separate from the two-dimensional transformation established here at level
13.

\section{Eta products and a theta identity}
For $1\le g\le6$, set
\begin{equation}\label{eq:eta}
 E_g(\tau)=q^{\frac{13}{2}B_2(g/13)}
 \prod_{\substack{n>0\\n\equiv\pm g\ (13)}}(1-q^n),
 \qquad B_2(x)=x^2-x+\frac16.
\end{equation}
These are generalized eta functions in the convention of Yang~\cite{Yang04,Yang09}. Their powers and suitable quotients are modular functions; the individual factors need not be invariant on $\Gamma_1(13)$.
Define
\begin{equation}\label{eq:G}
 G_0=(E_1E_5)^{-1},\quad G_1=(E_2E_3)^{-1},\quad G_2=(E_4E_6)^{-1}.
\end{equation}
Direct evaluation of the Bernoulli exponents gives
\begin{equation}\label{eq:normalization}
 G_0=q^{-1/6}P_0,\qquad G_1=q^{-1/6}P_1,\qquad G_2=q^{5/6}P_2.
\end{equation}
Thus $H=(G_0,G_1)^{\mathsf T}$. These formulas specify the multipliers, rather than claiming invariance of $G_j$ on $\Gamma_1(13)$.

\begin{lemma}\label{lem:theta}
One has
\begin{equation}\label{eq:theta-relation}
 G_0=G_1+G_2,\qquad P_0P_1P_2=\frac{(q^{13};q^{13})_\infty}{(q;q)_\infty}.
\end{equation}
\end{lemma}
\begin{proof}
Write $\theta(x;Q)=(x,Q/x;Q)_\infty$ and use the same notation with several arguments for their product. In the normalization of~\cite{Koornwinder}, the Weierstrass identity reads
\begin{align*}
 &yu\,\theta(xy,x/y,vu,v/u;Q)
 +uv\,\theta(xu,x/u,yv,y/v;Q)\\
 &\hspace{35mm}+vy\,\theta(xv,x/v,uy,u/y;Q)=0.
\end{align*}
Set $Q=q^{13}$ and $(x,y,u,v)=(q,q^3,q^2,q^5)$. Writing $\theta_j=\theta(q^j;Q)$, and using
$\theta(z^{-1};Q)=-z^{-1}\theta(z;Q)$ and $\theta(Q/z;Q)=\theta(z;Q)$, gives
\[
 \theta_2\theta_3\theta_4\theta_6
 =\theta_1\theta_4\theta_5\theta_6
   +q\theta_1\theta_2\theta_3\theta_5.
\]
Divide by $\theta_1\cdots\theta_6$ to obtain $P_0=P_1+qP_2$, and apply~\eqref{eq:normalization}. The product formula follows because the cosets partition all nonzero residues modulo 13.
\end{proof}

Set
\begin{equation}\label{eq:t}
 t=\frac{G_2}{G_1}=\frac{E_2E_3}{E_4E_6}=q\frac{P_2}{P_1}.
\end{equation}
For reference, its initial expansion is
\begin{equation}\label{eq:t-expansion}
 t=q-q^3-q^4+q^5+q^6-2q^{11}-q^{12}+O(q^{13}).
\end{equation}

\section{Periods and the Fricke image}
\begin{lemma}\label{lem:periods}
The periods are the three roots of $X^3+X^2-4X+1$, and
\begin{equation}\label{eq:period-coordinates}
 \eta_1=2-2\eta_0-\eta_0^2,\qquad
 \eta_2=-3+\eta_0+\eta_0^2.
\end{equation}
In particular,
\begin{equation}\label{eq:abc}
 \alpha=\beta+\gamma,\quad \gamma^2+\alpha\beta=13,
 \quad \beta^2+\alpha\gamma=13,\quad \alpha\beta\gamma=13.
\end{equation}
\end{lemma}
\begin{proof}
Counting sums of elements in the cosets gives
\[
 \eta_0^2=4+\eta_1+2\eta_2,\quad
 \eta_1^2=4+2\eta_0+\eta_2,\quad
 \eta_2^2=4+\eta_0+2\eta_1.
\]
Together with $\eta_0+\eta_1+\eta_2=-1$, these yield
$\eta_0\eta_1+\eta_0\eta_2+\eta_1\eta_2=-4$ and
$\eta_0\eta_1\eta_2=-1$. They also give~\eqref{eq:period-coordinates}. Substitution proves~\eqref{eq:abc}.
\end{proof}

Let
\begin{equation}\label{eq:sT}
 \begin{split}
 s_0&=\sin(\pi/13)\sin(5\pi/13),\quad
 s_1=\sin(2\pi/13)\sin(3\pi/13),\\
 s_2&=\sin(4\pi/13)\sin(6\pi/13),\qquad
 T_j(\tau)=\prod_{n\ge1}\prod_{a\in C_j}(1-\zeta^a q^n).
 \end{split}
\end{equation}
We write $W\tau=-1/(13\tau)$.

\begin{lemma}\label{lem:fricke-image}
For $j=0,1,2$,
\begin{equation}\label{eq:fricke-image}
 G_j(W\tau)=\frac{q^{-1/6}}{4s_jT_j(\tau)}.
\end{equation}
The normalization of the multiplier in this formula does not depend on a special property of the prime 13.
\end{lemma}
\begin{proof}
For a rational row vector $a\notin\ZZ^2$, put $z=e^{2\pi i(a_1\tau+a_2)}$ and use the Siegel function $g_a$ with product convention
\[
 g_{(a_1,a_2)}(\tau)
 =-q^{B_2(a_1)/2}e^{\pi i a_2(a_1-1)}(1-z)
 \prod_{n\ge1}(1-q^nz)(1-q^n/z).
\]
Then $E_g(\tau)=-g_{(g/13,0)}(13\tau)$. The Siegel transformation formula~\cite{KL} gives
$g_a(S\tau)=\varepsilon(S)g_{aS}(\tau)$, where
$S=\left(\begin{smallmatrix}0&-1\\1&0\end{smallmatrix}\right)$ and the multiplier is independent of $a$. Consequently
\[
 E_g(W\tau)=2i\varepsilon(S)\sin(\pi g/13)\,
 q^{1/12}\prod_{n\ge1}(1-\zeta^gq^n)(1-\zeta^{-g}q^n).
\]
On the positive imaginary axis, both $E_g(W\tau)$ and the product following the sine are positive real numbers. Since $\varepsilon(S)$ has absolute value one, this forces $\varepsilon(S)=-i$. Multiplying two such formulas and taking the reciprocal proves~\eqref{eq:fricke-image}.
\end{proof}

\begin{lemma}\label{lem:gauss}
The constants in~\eqref{eq:sT} satisfy
\begin{equation}\label{eq:gauss}
 4s_0\alpha=4s_1\beta=4s_2\gamma=\sqrt{13}.
\end{equation}
\end{lemma}
\begin{proof}
Let $\chi$ be the quadratic character modulo 13, and put
\[
 g_j=\sum_{a\in C_j}\chi(a)\zeta^a,\qquad
 \mathfrak g=\sum_{a=1}^{12}\chi(a)\zeta^a.
\]
The identity $2\sin A\sin B=\cos(A-B)-\cos(A+B)$ gives
$4s_0=-g_1$, $4s_1=g_2$, and $4s_2=g_0$.
Multiplication in $\ZZ[\zeta]$, reducing by
$\Phi_{13}(Z)=1+Z+\cdots+Z^{12}$, gives
\[
 -g_1\alpha=g_2\beta=g_0\gamma=\mathfrak g,
 \qquad \mathfrak g^2=13.
\]
These are finite polynomial identities; the accompanying verifier checks each directly in the cyclotomic ring. To fix the sign, observe that $\eta_1>0$ and $\eta_2<0$ from their cosine expressions, so $4s_0\alpha>0$. Thus $\mathfrak g=\sqrt{13}$.
\end{proof}

In particular $S_{13}=\frac14\left(\begin{smallmatrix}s_2^{-1}&s_1^{-1}\\s_0^{-1}&-s_2^{-1}\end{smallmatrix}\right)$.
The cubic polynomial in Lemma~\ref{lem:periods} is irreducible over $\QQ$, and its period field is the fixed field of $C_0$. It follows that $[K:\QQ]=3$. If a nonzero entry such as $\alpha/\sqrt{13}$ belonged to $K$, then $\sqrt{13}\in K$, which is impossible by the tower law. Since both $K$ and $\sqrt{13}=\mathfrak g$ belong to $\QQ(\zeta)^+$, their compositum is that degree-six real cyclotomic field. This proves the field assertions made in the introduction.

\subsection{Expanding the two sides}\label{subsec:expansion}
Put $Q=e^{-2\pi i/(13\tau)}$. In components, \eqref{eq:main} asserts
\begin{align*}
 Q^{-1/6}P_0(Q)&=\frac{q^{-1/6}}{\sqrt{13}}
                  \bigl(\gamma P_0(q)+\beta P_1(q)\bigr),\\
 Q^{-1/6}P_1(Q)&=\frac{q^{-1/6}}{\sqrt{13}}
                  \bigl(\alpha P_0(q)-\gamma P_1(q)\bigr).
\end{align*}
Here $Q^{-1/6}=e^{\pi i/(39\tau)}$, in accordance with the fixed
exponential convention. Since $Q$ and $q$ are different nomes, comparison
of coefficients first requires the left side to be expressed in $q$.
Lemmas~\ref{lem:fricke-image} and~\ref{lem:gauss} give, independently of
Theorem~\ref{thm:main},
\[
 H(W\tau)=\frac{q^{-1/6}}{\sqrt{13}}
 \begin{pmatrix}\alpha T_0(\tau)^{-1}\\\beta T_1(\tau)^{-1}\end{pmatrix}.
\]
The defining products and the quartics~\eqref{eq:quartic} give
\begin{align*}
 P_0(q)&=1+q+q^2+q^3+q^4+2q^5+O(q^6),\\
 P_1(q)&=1+q^2+q^3+q^4+q^5+O(q^6),\\
 T_j(\tau)^{-1}&=1+\eta_jq+q^2+q^3+q^4+(1+\eta_j)q^5+O(q^6)
 \quad(j=0,1).
\end{align*}
These coefficients can be obtained by finite multiplication in
$\ZZ[\eta_0]$. For example, the coefficient of $q^2$ in $T_j^{-1}$ is
$\eta_j^2+\eta_j-2-\eta_{j-1}=1$, by the period relations
(subscripts modulo three). The same relations give
$\alpha\eta_0=\gamma$ and $\beta\eta_1=\alpha$.
Together with $\alpha=\beta+\gamma$, they show that both
$H(W\tau)$ and $S_{13}H(\tau)$ have the expansion
\[
 \frac{q^{-1/6}}{\sqrt{13}}
 \begin{pmatrix}
 \alpha+\gamma q+\alpha(q^2+q^3+q^4)+(\alpha+\gamma)q^5+O(q^6)\\
 \beta+\alpha q+\beta(q^2+q^3+q^4)+(\alpha+\beta)q^5+O(q^6)
 \end{pmatrix}.
\]
This calculation makes the first matching coefficients explicit; it is not
by itself a proof of the full transformation. The divisor and
scalar-lifting arguments in Sections~4--6 establish the identity for all
$\tau\in\HH$.

\section{Modularity and cusp divisors}
We now specify precisely the modular functions used in the proof. For odd $N$, Yang's criterion~\cite[Proposition 2]{Yang09}, also~\cite[Corollary 3]{Yang04}, states that
\begin{equation}\label{eq:yang}
 \sum_g r_g\equiv0\pmod{12},\qquad
 \sum_g g^2r_g\equiv0\pmod N
\end{equation}
are sufficient for $\prod_g E_g^{r_g}$ to be a modular function on $\Gamma_1(N)$.
For $t=E_2E_3/(E_4E_6)$ these two sums are $0$ and $-39$; for
$G_1^6=E_2^{-6}E_3^{-6}$ they are $-12$ and $-78$. Hence both functions are invariant on $\Gamma_1(13)$, are nonvanishing on $\HH$, and have divisors supported at cusps.

There are two families of six cusps. The cusps represented by $a/c$ with $13\mid c$ have width one and are indexed by the nonzero numerator $a$ modulo sign; write them $c_a$ for $1\le a\le6$, with $c_1=\infty$. The other six have width 13. In local parameters on $X_1(13)$ the generalized eta orders are~\cite[Proposition 3]{Yang09}
\begin{equation}\label{eq:cusp-order}
 \ord_{a/c}(E_g)=
 \begin{cases}
 \dfrac{13}{2}B_2(\{ag/13\}),&13\mid c,\\[2pt]
 \dfrac1{12},&13\nmid c.
 \end{cases}
\end{equation}
Here the orders of individual $E_g$ can be fractional; their linear combinations give the integral orders of the invariant products just established.

For completeness, the widths and the factors in~\eqref{eq:cusp-order} can also be seen directly. If
$\sigma=\left(\begin{smallmatrix}a&b\\c&d\end{smallmatrix}\right)$ and $e=\gcd(c,13)$, factor
\[
 \begin{pmatrix}13&0\\0&1\end{pmatrix}\sigma
 =\sigma_1\begin{pmatrix}e&B\\0&13/e\end{pmatrix},\qquad \sigma_1\in\mathrm{SL}_2(\ZZ).
\]
The first component of $(g/13,0)\sigma_1$ is $ga/e$, and the transformed variable is $(e^2\tau+eB)/13$. The Siegel product therefore gives~\eqref{eq:cusp-order} after using the cusp width $13/e$.

Set
\[
 D=[c_2]+[c_3],\qquad \mathcal C=\text{the sum of all twelve cusps}.
\]
The required orders are
\begin{equation}\label{eq:order-table}
 \begin{array}{c|rrrrrr|c}
 &c_1&c_2&c_3&c_4&c_5&c_6&\text{each other cusp}\\\hline
 t&1&-1&-1&0&1&0&0\\
 G_1^6&-1&5&5&-1&-1&-1&-1
 \end{array}
\end{equation}
Thus
\begin{equation}\label{eq:divisors}
 \divisor(t)=[c_1]+[c_5]-D,\qquad
 \divisor(G_1^6)=-\mathcal C+6D.
\end{equation}
In particular $t$ has degree two.

The matrix $W=\left(\begin{smallmatrix}0&-1\\13&0\end{smallmatrix}\right)$ normalizes $\Gamma_1(13)$, since
\[
 W\begin{pmatrix}a&b\\13c&d\end{pmatrix}W^{-1}
 =\begin{pmatrix}d&-c\\-13b&a\end{pmatrix}.
\]
It interchanges the two cusp families. By Lemmas~\ref{lem:fricke-image} and~\ref{lem:gauss},
\begin{equation}\label{eq:tW}
 \ttil:=\frac{T_1}{T_2}=\frac{\beta}{\gamma}\,t\circ W.
\end{equation}
Consequently $\ttil$ is another modular unit of degree two. Its divisor is supported on the other cusp family, disjoint from the support of $\divisor(t)$.

\section{A five-coefficient identity}
Define
\begin{equation}\label{eq:F}
 F=\gamma\beta\ttil+\gamma\alpha\ttil t+\beta^2t-\beta\gamma.
\end{equation}
This is an invariant modular function on $\Gamma_1(13)$, holomorphic on $\HH$. Its polar divisor has degree at most four: at the first cusp family $\ord(\ttil)=0$, so only the two simple poles of $t$ contribute; at the second family the two simple poles of $\ttil$ contribute. This bounds poles of $t\ttil$ as well, because the two supports are disjoint.

Put
\begin{equation}\label{eq:Xi}
 \Xi=FT_2P_1
 =\gamma\beta(T_1-T_2)P_1
  +\gamma\alpha T_1qP_2+\beta^2T_2qP_2.
\end{equation}
Every coefficient lies in $\ZZ[\eta_0]$. Indeed, with subscripts modulo three,
\begin{equation}\label{eq:quartic}
 \prod_{a\in C_j}(1-\zeta^a X)
 =1-\eta_jX+(2+\eta_{j-1})X^2-\eta_jX^3+X^4,
\end{equation}
and $T_j$ is the product of these quartics at $X=q^n$.

The following table gives all coefficients needed for the proof. A triple $(a,b,c)$ means $a+b\eta_0+c\eta_0^2$. Set
$U=(T_1-T_2)P_1$, $V=T_1qP_2$, and $Z=T_2qP_2$.
\begin{equation}\label{eq:coefficient-table}
\begin{array}{c|rrr}
 n &[q^n]U &[q^n]V &[q^n]Z\\\hline
 0&(0,0,0)&(0,0,0)&(0,0,0)\\
 1&(-5,3,2)&(1,0,0)&(1,0,0)\\
 2&(-7,6,3)&(-2,2,1)&(3,-1,-1)\\
 3&(-22,15,9)&(0,3,1)&(7,-3,-2)\\
 4&(-41,30,17)&(-3,7,3)&(14,-5,-4)
\end{array}
\end{equation}
Using
\[
 \alpha=(5,-3,-2),\quad\beta=(3,0,-1),\quad\gamma=(2,-3,-1),
 \qquad \eta_0^3=-\eta_0^2+4\eta_0-1,
\]
each row gives
$\gamma\beta[q^n]U+\gamma\alpha[q^n]V+\beta^2[q^n]Z=0$.
Only the factors with $n\le4$ in~\eqref{eq:quartic} can contribute, so this is an explicit finite multiplication. The supplementary script checks the table and, redundantly, all coefficients of $\Xi$ through $q^{24}$.

Since $T_2P_1=1+O(q)$, we have $\ord_\infty(F)\ge5$. The cusp width is one. A nonzero meromorphic function on a compact curve has equal degrees of its zero and pole divisors. The bound of four on the polar degree therefore implies $F=0$. By~\eqref{eq:tW} and~\eqref{eq:F},
\begin{equation}\label{eq:projective}
 t(W\tau)=\frac{\gamma-\beta t(\tau)}{\beta+\alpha t(\tau)}.
\end{equation}
This proves the projective transformation, but the common scalar remains to be determined.

\section{Lifting the projective transformation}
\begin{proof}[Proof of Theorem~\ref{thm:main}]
The numerator and denominator in~\eqref{eq:projective} have no common zero, since $\beta^2+\alpha\gamma=13$. At a pole of $t$ their ratio has a finite value. The polar divisor of $t\circ W$ is $W(D)$, where the involution identifies pushforward and pullback. It follows that
\begin{equation}\label{eq:denominator-divisor}
 \divisor(\beta+\alpha t)=W(D)-D.
\end{equation}
In particular this function does not vanish on $\HH$.

Now define the nonvanishing holomorphic function
\[
 R(\tau)=\frac{\sqrt{13}\,G_1(W\tau)}{(\beta+\alpha t(\tau))G_1(\tau)}.
\]
Its sixth power is an invariant modular function on $\Gamma_1(13)$, by the modularity of $G_1^6$ and $t$, and the normalization of the group by $W$. Equations~\eqref{eq:divisors} and~\eqref{eq:denominator-divisor} give
\begin{align*}
 \divisor(R^6)
 &=(-\mathcal C+6W(D))-(-\mathcal C+6D)-6(W(D)-D)\\
 &=0.
\end{align*}
Hence $R^6$ is constant on $X_1(13)$. Because $R$ is holomorphic on the connected domain $\HH$, it is constant as well. As $\tau\to i\infty$,
\[
 G_1(W\tau)\sim\frac{\beta}{\sqrt{13}}q^{-1/6},\qquad
 G_1(\tau)\sim q^{-1/6},\qquad t(\tau)\to0,
\]
by~\eqref{eq:fricke-image} and~\eqref{eq:gauss}. Therefore $R=1$, and
\[
 G_1(W\tau)=\frac{\beta+\alpha t}{\sqrt{13}}G_1
 =\frac{\alpha G_0-\gamma G_1}{\sqrt{13}}.
\]
Using~\eqref{eq:projective} gives
$G_2(W\tau)=(\gamma-\beta t)G_1/\sqrt{13}$, and adding these two identities yields
$G_0(W\tau)=(\gamma G_0+\beta G_1)/\sqrt{13}$.
This is~\eqref{eq:main}. Finally~\eqref{eq:abc} gives $S_{13}^2=I$, and~\eqref{eq:normalization} gives the translation law.
\end{proof}

\begin{corollary}\label{cor:twisted}
Let $\ell_0=\gamma P_0+\beta P_1$, $\ell_1=\alpha P_0-\gamma P_1$, and $\ell_2=-\beta P_0+\alpha P_1$. Then
\[
 T_0\ell_0=\alpha,\qquad T_1\ell_1=\beta,\qquad T_2\ell_2=\gamma.
\]
\end{corollary}
\begin{proof}
Compare each transformed component with~\eqref{eq:fricke-image} and use~\eqref{eq:gauss}.
\end{proof}

\begin{remark}
The projective transformation is compatible with the geometry of $X_1(13)$. This curve has genus two: its effective index is 84, it is torsion-free, and it has twelve cusps. Thus the degree-two function $t$ is a hyperelliptic coordinate. An automorphism of the curve acts fractionally linearly on such a coordinate. The argument above determines this action explicitly and then supplies the additional scalar information needed for the eta-product vector.
\end{remark}

\section{What holds at every cubic-residue prime}
Let $p\equiv1\pmod3$ be prime, put $m=(p-1)/6$, and let $C_j$ be the three cubic-residue cosets. Write
\[
 C_j^+=\{a\in C_j:1\le a<p/2\},\qquad
 e_j=\frac p2\sum_{a\in C_j^+}B_2(a/p),
\]
\[
 G_j=q^{-e_j}\prod_{\substack{n>0\\n\bmod p\in C_j}}(1-q^n)^{-1}.
\]
Here $|C_j^+|=m$. Define
\[
 s_j=\prod_{a\in C_j^+}\sin(\pi a/p),\qquad
 T_j=\prod_{n\ge1}\prod_{a\in C_j}(1-\zeta_p^a q^n),\qquad
 \zeta_p=e^{2\pi i/p}.
\]

\begin{proposition}\label{prop:general}
For every such prime,
\begin{equation}\label{eq:general}
 \sum_{j=0}^2 e_j=\frac{1-p}{24},\qquad
 G_0G_1G_2=\frac{\eta(p\tau)}{\eta(\tau)},\qquad
 G_j\left(-\frac1{p\tau}\right)=\frac{q^{-m/12}}{2^ms_jT_j(\tau)}.
\end{equation}
In particular
\[
 (G_0G_1G_2)\left(-\frac1{p\tau}\right)
 =\frac{1}{\sqrt p\,(G_0G_1G_2)(\tau)}.
\]
\end{proposition}
\begin{proof}
The half-cosets partition $\{1,\ldots,(p-1)/2\}$, and the Bernoulli multiplication formula gives
$\sum_{a=1}^{(p-1)/2}B_2(a/p)=(1-p)/(12p)$.
This proves the exponent sum; multiplying the defining products proves the eta quotient. Its transformation follows from $\eta(-1/\tau)=\sqrt{-i\tau}\eta(\tau)$. The proof of Lemma~\ref{lem:fricke-image}, with $p$ in place of 13 and $m$ factors, proves the final formula. The same positivity argument determines $\varepsilon(S)=-i$ at every prime.
\end{proof}

\begin{theorem}\label{thm:membership}
Let $K_p$ be the cyclic cubic subfield of $\QQ(\zeta_p)$. Then
\[
 \nu_j:=\frac{\sqrt p}{2^ms_j}\in K_p\qquad(j=0,1,2)
\]
for every prime $p\equiv1\pmod3$.
\end{theorem}
\begin{proof}
Choose a primitive root $r$ modulo $p$, put $b=r^3$, and let $h$ be the inverse of 2 modulo $p$. Thus $b$ generates the cubic-residue subgroup, has order $2m$, and satisfies $b^m=-1$. Set
\[
 D_j=\prod_{a\in C_j^+}(\zeta_p^{ha}-\zeta_p^{-ha}),\qquad
 \mathfrak g_p=\sum_{a=1}^{p-1}\left(\frac ap\right)\zeta_p^a.
\]
Multiplication by $b$ cyclically permutes the $m$ pairs $\{\pm a\}$ in a coset. The product of the signs picked up in this cycle is $-1$, since $b^m=-1$. Consequently the automorphism $\sigma_b(\zeta_p)=\zeta_p^b$ satisfies
$\sigma_b(D_j)=-D_j$. Also $b$ is a nonsquare, so
$\sigma_b(\mathfrak g_p)=-\mathfrak g_p$ by a change of summation variable. The quotient $\mathfrak g_p/D_j$ is therefore fixed by the cubic-residue subgroup and belongs to $K_p$.

Complex conjugation multiplies $D_j$ by $(-1)^m$ and $\mathfrak g_p$ by $(-1)^{(p-1)/2}=(-1)^m$, so the quotient is real. Moreover
\[
 |D_j|=2^ms_j,\qquad |\mathfrak g_p|^2=p.
\]
For the second equality, substitute $a=tb$ in the double sum for the squared absolute value; the inner sum $\sum_{b\ne0}\zeta_p^{b(t-1)}$ is $p-1$ at $t=1$ and $-1$ otherwise. Hence
$\mathfrak g_p/D_j=\pm\nu_j$, proving the assertion.
\end{proof}

\begin{remark}\label{rem:scope}
Membership in $K_p$ is different from being a difference of two cubic Gaussian periods. At 13 the stronger identities~\eqref{eq:gauss} hold; Theorem~\ref{thm:membership} does not assert their analogue at every prime. Neither the stronger representation nor a repeated normalizing exponent is asserted here to be necessary for all linear Fricke transformation laws.
\end{remark}

\section*{Data, code, and disclosure}
No empirical datasets were used. The supplementary verifier
\texttt{verify\_level13.py} uses only the Python standard library. It
checks the period and Gauss-sum identities, the quartics~\eqref{eq:quartic},
both modularity congruences, both rows of~\eqref{eq:order-table}, the
five-coefficient table~\eqref{eq:coefficient-table}, the longer redundant
coefficient check through $q^{24}$, and the expansion~\eqref{eq:t-expansion}.
All arithmetic is exact. These finite checks do not replace the analytic
modularity, divisor, or scalar-lifting arguments in the text.

Two large language models were used in preparing this paper: Claude
Fable 5.1 (Anthropic) and ChatGPT with the GPT-5.6 Sol model (OpenAI).
Their assistance included the reduction of Theorem~\ref{thm:main} to the
identity $F\equiv0$, the cusp-divisor and Gauss-sum computations, the exact
verification of the coefficient table, the literature comparison in
Section~1, preparation of the verification code, and manuscript revision.
Responsibility for the mathematical statements, proofs, and computations
remains with the author. The author received no external funding and
declares no conflict of interest.


\begin{thebibliography}{99}
\bibitem{Brunault} F.~Brunault, \emph{On the Mahler measure associated to $X_1(13)$}, preprint (2015), arXiv:1503.04631.
\bibitem{Koornwinder} T.~H.~Koornwinder, \emph{On the equivalence of two fundamental theta identities}, Anal. Appl. (Singap.) \textbf{12} (2014), 711--725, doi:10.1142/S0219530514500559.
\bibitem{KL} D.~S.~Kubert and S.~Lang, \emph{Modular Units}, Grundlehren der mathematischen Wissenschaften 244, Springer, New York, 1981.
\bibitem{Mizuno} Y.~Mizuno, \emph{Remarks on Nahm sums for symmetrizable matrices}, Ramanujan J. \textbf{66} (2025), Article 62, doi:10.1007/s11139-025-01033-6.
\bibitem{Ono} K.~Ono, \emph{Modularity from $q$-series}, preprint (2025), arXiv:2509.20316, version 5 (2026).
\bibitem{VP} K.~R.~Vasuki and N.~Prabhu, \emph{Level-13 theta function identities of Ramanujan and applications}, Acta Arith. \textbf{219} (2025), 81--99, doi:10.4064/aa240728-7-10.
\bibitem{WW2} B.~Wang and L.~Wang, \emph{Proofs of Mizuno's conjectures on generalized rank two Nahm sums}, Trans. Amer. Math. Soc. \textbf{379} (2026), no.~1, 459--486, doi:10.1090/tran/9496.
\bibitem{WW3} B.~Wang and L.~Wang, \emph{Proofs of Mizuno's conjectures on rank three Nahm sums of index $(1,2,2)$}, Adv. Math. \textbf{477} (2025), Article 110368, doi:10.1016/j.aim.2025.110368.
\bibitem{WZ} L.~Wang and H.~Zhang, \emph{Modularity of some Nahm sums as vector-valued functions}, Ramanujan J. \textbf{68} (2025), Article 48, doi:10.1007/s11139-025-01194-4.
\bibitem{Yang04} Y.~Yang, \emph{Transformation formulas for generalized Dedekind eta functions}, Bull. London Math. Soc. \textbf{36} (2004), 671--682, doi:10.1112/S0024609304003510.
\bibitem{Yang09} Y.~Yang, \emph{Modular units and cuspidal divisor class groups of $X_1(N)$}, J. Algebra \textbf{322} (2009), 514--553, doi:10.1016/j.jalgebra.2009.04.012.
\end{thebibliography}
\end{document}